\documentclass[journal, 10 pt, onecolumn,draftcls,a4paper]{ieeeconf}
\usepackage[latin1]{inputenc}
\usepackage{graphicx}
\usepackage{amsfonts}
\usepackage{amssymb}
\usepackage{amsmath}
\usepackage{psfrag}
\usepackage{color}
\usepackage{balance}

\IEEEoverridecommandlockouts                              
\newtheorem{lemma}{Lemma}
\newtheorem{theorem}{Theorem}

\newtheorem{proposition}{Proposition}

\newtheorem{definition}{Definition}

\def\rr{\mathbb{R}}
\def\sr{\tilde{r}}

\def\nn{\mathbb{N}}
\def\cM{\mathcal{M}}
\def\cR{\mathcal{R}}

\newcommand{\ds}{\displaystyle}
\newcommand{\vect}[2]{\begin{bmatrix}#1\\#2\end{bmatrix}}

\graphicspath{{./pics/}}

\title{\LARGE \bf An improved lion strategy for the lion and man problem}
\author{Marco Casini, Andrea Garulli
\thanks{M.~Casini and A.~Garulli are with the Dipartimento di
Ingegneria dell'Informazione e Scienze Matematiche, Universit\`a degli Studi di Siena, via Roma~56, 53100 Siena, Italy. E-mail:
casini@dii.unisi.it,~garulli@dii.unisi.it.}}

\begin{document}
\maketitle \thispagestyle{empty} \pagestyle{empty}

\begin{abstract}
In this paper, a novel lion strategy for David Gale's lion and man problem is proposed. The devised approach enhances a popular strategy proposed by
Sgall, which relies on the computation of a suitable ``center''. The key idea of the new strategy is to update the center at each move, instead of
computing it once and for all at the beginning of the game. Convergence of the proposed lion strategy is proven and an upper bound on the game length is
derived, which dominates the existing bounds.
\end{abstract}


\section{Introduction}\label{sec:intro}

Pursuit-evasion games have attracted the interest of researchers for long time, both for the intriguing mathematics they require (see \cite{nahin2012}
for a nice introduction), and for the variety of applications they find in different contexts, ranging from mobile robotics to surveillance, resource
harvesting, network security and many others. When the game is played in a limited environment, problems become even more challenging. Among the huge
number of different formulations (an extensive survey is provided in \cite{chung2011}), two main classifications can be singled out, concerning
respectively time and space being continuous or discrete. If continuous time is assumed, it is well known that an evader can indefinitely escape a single
pursuer travelling at the same velocity, even in very simple continuous environments, like a circle \cite{littlewood1986}. On the other hand, if time is
discrete, the pursuer can capture the evader in finite time, in many situations of interest. This has generated a rich literature, considering different
assumptions on the number of pursuers, the structure of the environment and the information available to the players (see, e.g.,
\cite{isler2005,bopardikar2008,bhadauria2012,ames2015,aleem2015} and references therein).

A fundamental problem at the basis of the above literature is the so-called \emph{lion and man problem}, whose formulation is ascribed to Gale (see
problem 31 in \cite{guy1995}). A lion and a  man move alternately in the positive quadrant of the plane, travelling a distance of at most one unit at
each move. It is known that the lion can catch the man in finite time, provided that his initial coordinates are componentwise larger than those of the
man. Nevertheless, the optimal lion strategy is still an open problem. In \cite{sgall01}, Sgall has proposed a nice strategy for the lion, which
guarantees capture in finite time. Moreover, he has given an upper bound on the capture time which is achieved for some specific initial conditions.
Sgall's strategy is based on the definition of a fixed \emph{center}, depending on the initial lion and man positions: then, the lion always keeps on the
line connecting the center to the man's position, until capture occurs. This strategy has been used in several mobile robotics application, as reported
in the tutorial \cite{noori2016}. In particular, a slight variation of the solution proposed in \cite{sgall01} is adopted iteratively in
\cite{isler2005}, where it is instrumental to devise a strategy for two pursuers to capture an evader in simply connected polygonal environments. A
similar variation is employed in \cite{noori2014}, when dealing with pursuer evasion games in monotone polygons with line-of-sight visibility.

In this paper, a new lion strategy is proposed for the lion and man problem, which improves the one proposed in \cite{sgall01}. The main idea is to
compute a new center at each move, in order to enhance the advantage gained by the lion in a single step. This turns out to be effective also on the
whole, as it allows one to derive an upper bound on the maximum number of moves required to guarantee capture, which dominates the one given in
\cite{sgall01}.

The paper is organized as follows. In Section \ref{formulation}, the lion and man problem is formulated. The solution proposed in \cite{sgall01} is
reviewed in Section \ref{sec:FCLS}, along with some useful properties. The new strategy is introduced in Section \ref{sec:MCLS} and its convergence
properties are derived in Section \ref{sec:convergence}. Concluding remarks are reported in Section \ref{sec:conclusion}.

\section{Problem formulation}\label{formulation}
The notation adopted in the paper is standard. Let $\nn$ and $\rr^n_+$ denote the set of all natural numbers and the $n$-dimensional Euclidean space of
non-negative numbers, respectively. Let $\lceil x\rceil$ be the smallest integer greater or equal to $x$. A row vector with elements $v_1,\ldots,v_n$ is
denoted by $V=[v_1,\ldots,v_n]$, while $V'$ is the transpose of $V$.

In this paper, we consider the version of the \emph{lion and man problem} formulated by David Gale. Two players, a man and a lion, can move in the first
quadrant of the Cartesian plane. Time is assumed discrete, while space is continuous. At each round (hereafter called \emph{time}) both players are
allowed to move to any point inside the non-negative quadrant, with distance less or equal to a given radius $r$ from their current position. Hereafter,
it will be assumed $r=1$ without loss of generality. The man moves first, after that the lion moves. Let us denote by $M_t\in\rr^2_+$ and $L_t\in\rr^2_+$
the man and lion position at time $t$, respectively. Hence, $\|M_{t+1}-M_t\|\le1$, $\|L_{t+1}-L_t\|\le1$. The game ends (lion wins) if the lion moves
\emph{exactly} to the man position. If the man can escape indefinitely from the lion, the man wins. It is assumed that the initial man coordinates are
strictly smaller than the corresponding lion coordinates, otherwise it is straightforward to observe that the man wins the game by moving straight up or
right.

\section{Fixed Center Lion Strategy}\label{sec:FCLS}

Before introducing the proposed lion strategy, let us recall the one devised in \cite{sgall01}, hereafter referred to as \emph{Fixed Center Lion Strategy
(FCLS)}. If the initial man coordinates are strictly smaller than the corresponding lion coordinates, the FCLS allows the lion to capture the man in a
finite number of moves, for any man strategy. In order to state the FCLS, the following definition is needed.

\begin{definition}\label{def:center0}
Let $M_0$ and $L_0$ be the man and lion position at time $0$, respectively. Let $C_0=[x_0,\,y_0]'\in\rr^2_+$ be the point satisfying
\begin{enumerate}
  \item $C_0=L_0+\eta(L_0-M_0)$, with $\eta>0$\ ;
  \item $\|C_0-L_0\|=\max\{x_0,y_0\}\ .$
\end{enumerate}
Then $C_0$ is called the \emph{center} of the FCLS.
\end{definition}

\medskip
\emph{Fixed Center Lion Strategy}

Let $C_0$ be the center of the FCLS. Such a center is fixed and does not change during the game. At a given time $t$, let the man move from $M_t$ to
$M_{t+1}$. The lion adopts the following strategy:
\begin{itemize}
  \item if $\|M_{t+1}-L_t\|\le 1$, then the lion moves to $M_{t+1}$ and catches the man;
  \item otherwise, the lion moves to a point on the line connecting $M_{t+1}$ to $C_0$ with unitary distance from $L_t$. Between the two points satisfying
  such a condition, he chooses the one farther from $C_0$.
\end{itemize}

Let $C_0=[x_0,\,y_0]'$ and denote by $r_0$ and $m_0$ the greatest and smallest element of $C_0$, respectively, i.e., $r_0=\max\{x_0,y_0\}$ and
$m_0=\min\{x_0,y_0\}$. Let us denote by $N_{max}^{FCLS}$ the upper bound derived in \cite{sgall01} on the maximum number of moves needed by the lion to
catch the man by using FCLS. Let us recall some results proved in \cite{sgall01} which will be useful in the following.

\begin{proposition}\label{prop:sgall_r1_max}
Let the lion play the FCLS. At every $t$, one has
\begin{enumerate}
  \item[i)~]$\|L_t-C_0\|^2+1 \le \|L_{t+1}-C_0\|^2 \le \|C_0\|^2$\ ;
  \item[ii)~]both elements of $C_0-L_{t}$ are strictly positive.
\end{enumerate}
\end{proposition}

\begin{proposition}\label{prop:sgall_max_move}
Let the lion play the FCLS. Then, the lion captures the man in a number of moves equal at most to
\begin{equation}\label{eq:FCLS_max_move}
N_{max}^{FCLS}\!=\!\left\lceil \|C_0\|^2-\|C_0\!-\!L_0\|^2\right\rceil\!=\!\left\lceil r_0^2+m_0^2-r_0^2\right\rceil\!=\!\left\lceil m_0^2\right\rceil
\end{equation}
\end{proposition}
The bound in \eqref{eq:FCLS_max_move} has been proved to be tight whenever the lion and the man start sufficiently close to each other. In this case, the
optimal strategy for the man is to move orthogonally w.r.t. the line connecting him to $C_0$.

\section{Novel lion strategy}\label{sec:MCLS}

In this section, the proposed lion strategy, hereafter referred to as \emph{Moving Center Lion Strategy (MCLS)}, is introduced. The main difference
between MCLS and FCLN regards the computation of the center. While in FCLS the center in computed once and for all at the beginning of the game, in the
proposed strategy it is updated at each move, and then it is used to compute the lion move.

Before describing the devised lion strategy, let us introduce the following definitions which will be used throughout the paper.

\begin{definition}\label{def:center}
Let $M_t$ and $L_t$ be the man and lion position at time $t$, respectively. Let $C_t=[x_t,\,y_t]'\in\rr^2_+$ be the point satisfying
\begin{enumerate}
  \item $C_t=L_t+\eta(L_t-M_t)$, with $\eta>0$
  \item $\|C_t-L_t\|=\max\{x_t,y_t\}\ .$
\end{enumerate}
Then $C_t$ is called the \emph{center} of the MCLS at time $t$.
\end{definition}

\begin{definition}\label{def:raggi}
At a given time $t$, let us define the following quantities:
\begin{enumerate}
  \item $r_t=\max\{x_t,y_t\}=\|C_t-L_t\|$
  \item $m_t=\min\{x_t,y_t\}$
  \item $\sr_{t+1}=\|L_{t+1}-C_t\|$\ .
\end{enumerate}
\end{definition}

\medskip
\emph{Moving Center Lion Strategy}

At a given time $t$, let the man move from $M_t$ to $M_{t+1}$. The lion moves according to the following strategy:
\begin{itemize}
  \item compute the center $C_t$, based on man and lion position at time $t$, according to Definition~\ref{def:center};
  \item if $\|M_{t+1}-L_t\|\le 1$, then the lion moves to $M_{t+1}$ and catches the man;
  \item otherwise, the lion moves to a point on the line connecting $M_{t+1}$ to $C_t$ with unitary distance from $L_t$. Between the two points satisfying
  such a condition, he chooses the one farther from $C_t$.
\end{itemize}

The following propositions hold.

\begin{proposition}\label{prop:MCLS_r1_max}
Let the lion play the MCLS. At every time $t$, one has
\begin{enumerate}
  \item[i)~] $\|L_t-C_t\|^2+1 \le \|L_{t+1}-C_t\|^2 \le \|C_t\|^2$\ ;
  \item[ii)~]both elements of $C_t-L_{t+1}$ are strictly positive\ ;
  \item[iii)~]the following inequalities hold
\begin{equation}\label{eq:r1t_max}
r_t^2+1 \le \sr_{t+1}^2 \le r_t^2+m_t^2\ .
\end{equation}
\end{enumerate}
\end{proposition}
\begin{proof}
Items $i)-ii)$ follow directly from Proposition~\ref{prop:sgall_r1_max}, because at each time $t$, the center $C_t$ is defined in the same way as $C_0$
in Definition~\ref{def:center0}; item $iii)$ stems from Definition~\ref{def:raggi} and item $i)$.
\end{proof}

\begin{proposition}\label{prop:1_move}
At a given time $t$, let $m_t\le 1$ and let the lion play the MCLS. Then, the lion captures the man in one move.
\end{proposition}
\begin{proof}
The proof is a direct consequence of Proposition~\ref{prop:sgall_max_move}, with $C_0=C_t$ and $L_0=L_t$.
\end{proof}

\section{Convergence analysis}\label{sec:convergence}

In this section, an upper bound to the maximum number of moves needed by the lion to catch the man, when using the MCLS, is derived. Moreover, it is
shown that such an upper bound is always smaller than the upper bound $N_{max}^{FCLS}$ provided by the FCLS.

Before stating the main result, the following lemmas are needed.

\begin{lemma}\label{lem:x1_y1}
Let $r_t=\|C_t-L_t\|$ and $\sr_{t+1}=\|C_t-L_{t+1}\|$. Then, $x_{t+1}<x_t$ and  $y_{t+1}<y_t$.
\end{lemma}
\begin{proof}
According to the MCLS, $L_{t+1}$ lies on the line connecting $C_t$ and $M_{t+1}$. On the other hand, by Definition~\ref{def:center}, $C_{t+1}$ lies on
the line joining $L_{t+1}$ and $M_{t+1}$. Hence, $C_t$, $L_{t+1}$ and $C_{t+1}$ belong to the same line, i.e.
$$C_{t+1}=C_t+(C_t-L_{t+1})\alpha\triangleq C_t+\vect{d_x}{d_y}~,~\alpha,d_x,d_y\in\rr.$$
By Proposition~\ref{prop:MCLS_r1_max}, $C_t-L_{t+1}$ has both coordinates strictly positive, i.e., $sign(d_x)=sign(d_y)=sign(\alpha).$ Hence, it is
sufficient to show
that $\alpha<0$.\\
By contradiction, assume $\alpha\ge0$ and let us define $d=\sqrt{d_x^2+d_y^2}\ge0$, see Fig.~\ref{fig:C1_min_C0}. Then, being $d_x\ge 0$ and $d_y\ge 0$,
one has
\begin{eqnarray}
\|C_{t+1}-L_{t+1}\|=\sr_{t+1}+d>r_t+d\ge r_t+\max\{d_x,d_y\}\nonumber\\ %
\hspace{-15mm}\ge\max\{x_t+d_x,y_t+d_y\}=\max\{x_{t+1},y_{t+1}\}\nonumber
\end{eqnarray}
where the strict inequality comes from \eqref{eq:r1t_max}. Since $\|C_{t+1}-L_{t+1}\|>\max\{x_{t+1},y_{t+1}\}$, $C_{t+1}$ cannot be a center, according
to Definition~\ref{def:center}.
\end{proof}

\begin{figure}[htb]
\centering %
\includegraphics[width=.6\columnwidth]{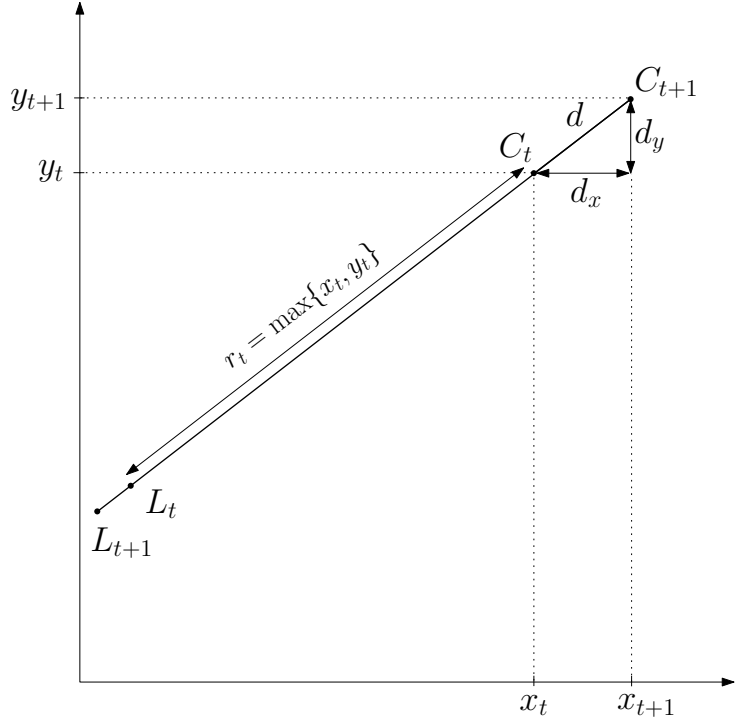}
\caption{Sketch of proof of Lemma~\ref{lem:x1_y1}.} \label{fig:C1_min_C0}
\end{figure}

\begin{lemma}\label{lem:punti_estremi}
Let $C_t$ and $\sr_{t+1}>r_t$ be given. Let $L_{t+1}$ satisfy $\|C_t-L_{t+1}\|=\sr_{t+1}$. Then,
\begin{eqnarray}
&&\hspace{-7mm}\widehat m_{t+1}\triangleq\sup_{L_{t+1}:\|C_t-L_{t+1}\|=\sr_{t+1}}~ m_{t+1}\nonumber\\
&&\hspace{-7mm}=\max\left\{ \ds m_t\frac{r_t\!-\!\sqrt{\sr_{t+1}^2-m_t^2}}{\sr_{t+1}\!-\!\sqrt{\sr_{t+1}^2-m_t^2}}\,,\,
r_t\frac{m_t\!-\!\sqrt{\sr_{t+1}^2-r_t^2}}{\sr_{t+1}\!-\!\sqrt{\sr_{t+1}^2-r_t^2}} \right\}.\nonumber\\\label{eq:widehat_m}
\end{eqnarray}
\end{lemma}
\medskip
\begin{proof}
Let us consider the case $r_t=x_t$ (the case $r_t=y_t$ is analogous).\\
Let $C_{t+1}=[x_{t+1},\,y_{t+1}]$ be the center at time $t+1$, and $\theta$ be the angle between the $x$ axis and the vector $C_t-L_{t+1}$. Notice that,
since $C_{t+1}$ lies on the line connecting  $L_{t+1}$ and $C_t$, $\theta$ is also the angle between the $x$ axis and the vector $C_{t+1}-L_{t+1}$ (see
Fig.~\ref{fig:dim_lem1}). It follows that $\theta\in[\overline\theta,\,\underline\theta]$ where
$$
\underline\theta=\arccos\left(\frac{r_t}{\sr_{t+1}}\right)\quad \mbox{and}\quad \overline\theta=\arcsin\left(\frac{m_t}{\sr_{t+1}}\right)\ .
$$

\begin{figure}[htb]
\centering %
\includegraphics[width=.6\columnwidth]{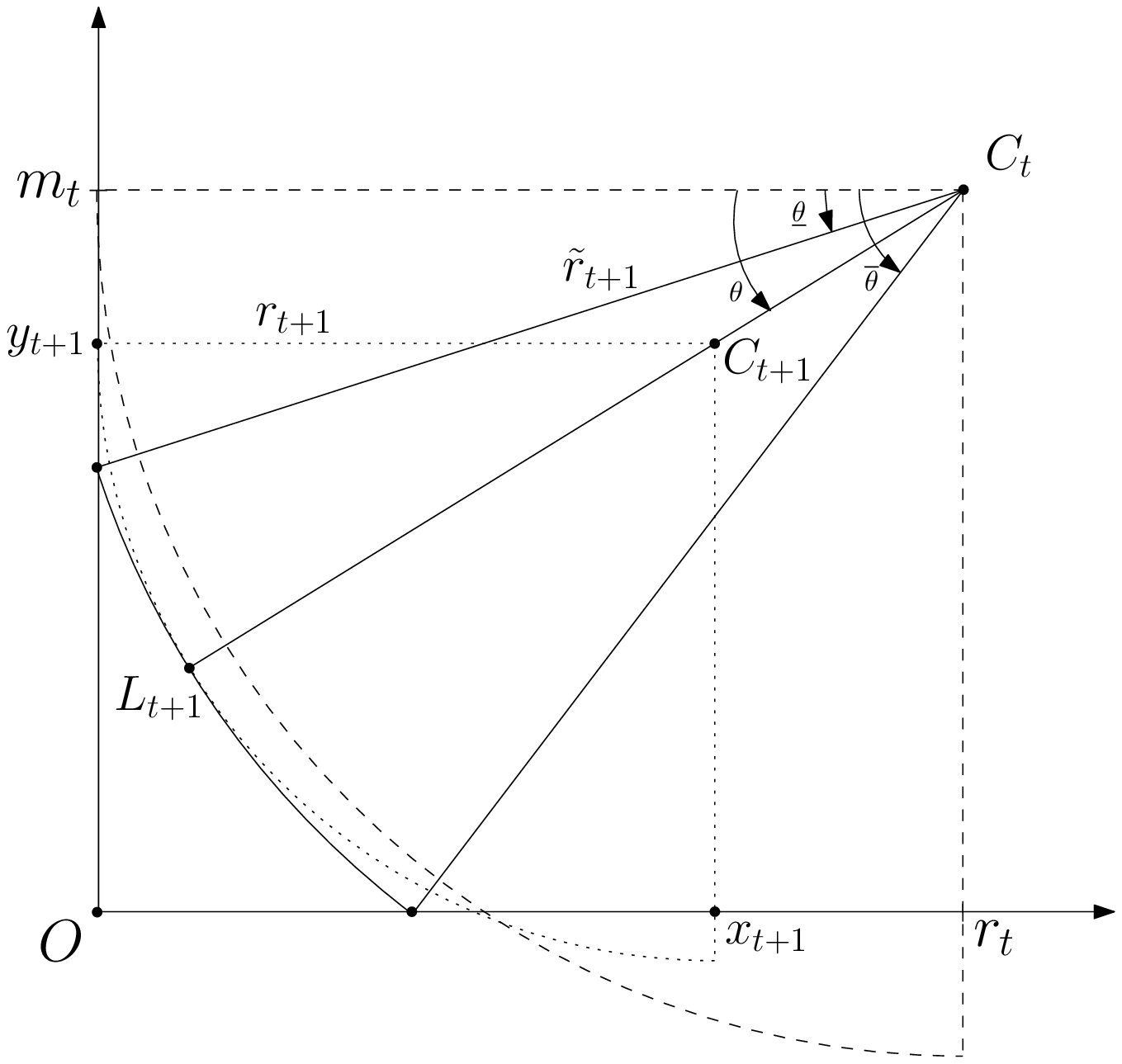}
\caption{Sketch of proof of Lemma~\ref{lem:punti_estremi}.} \label{fig:dim_lem1}
\end{figure}

Let us define $\delta=\sr_{t+1}-r_{t+1}$. By Definition~\ref{def:center} and Lemma~\ref{lem:x1_y1}, it turns out $\delta>0$. Moreover, one has
\begin{eqnarray}
x_{t+1}=&r_t-\delta\cos\theta&\triangleq f_x(\theta)\label{eq:lem_xt}\\
y_{t+1}=&m_t-\delta\sin\theta&\triangleq f_y(\theta)\label{eq:lem_yt}\ .
\end{eqnarray}
Let us define
$$
\cM(\theta)=\min\{f_x(\theta),\,f_y(\theta)\}
$$
and
$$
\cR(\theta)=\max\{f_x(\theta),\,f_y(\theta)\}\\ .
$$
Then, finding $\widehat m_{t+1}$ defined in \eqref{eq:widehat_m}, boils down to
\begin{eqnarray}
\widehat
m_{t+1}&=&\sup_{\theta\in[\overline\theta,\,\underline\theta]}\min\{x_{t+1},\,y_{t+1}\}\nonumber\\
&=&\sup_{\theta\in[\overline\theta,\,\underline\theta]}\min\{f_x(\theta),\,f_y(\theta)\} =
\sup_{\theta\in[\overline\theta,\,\underline\theta]}\cM(\theta)\ .~~~~\label{eq:min_fx_fy}
\end{eqnarray}
Let us analyze the case $\cM(\theta)=f_x(\theta)$, $\cR(\theta)=f_y(\theta)$. Notice that there exists at least one value of $\theta$ such that this
condition is satisfied. In fact, for $\theta=\underline\theta$ one has
\begin{eqnarray}
\cM(\underline\theta)&=&f_x(\underline\theta)=r_t-(\sr_{t+1}-r_{t+1})\cos(\underline\theta)\nonumber\\
&=&\frac{r_{t+1}\,r_t}{\sr_{t+1}}<r_{t+1}=f_y(\underline\theta)=\cR(\underline\theta)\ .\nonumber
\end{eqnarray}
Since $\cR(\theta)=f_y(\theta)=r_{t+1}$, by \eqref{eq:lem_yt}, one has
$$
r_{t+1}=m_t-(\sr_{t+1}-r_{t+1})\sin(\theta)
$$
which leads to
\begin{equation}\label{eq:new_radius}
r_{t+1}=\frac{m_t-\sr_{t+1}\sin(\theta)}{1-\sin(\theta)}\ .
\end{equation}
By substituting \eqref{eq:new_radius} into \eqref{eq:lem_xt}, after some algebra one gets
\begin{eqnarray}
\cM(\theta)&=&f_x(\theta)=r_t-(\sr_{t+1}-r_{t+1})\cos(\theta)\nonumber\\
&=& r_t-\frac{(\sr_{t+1}-m_t)\cos(\theta)}{1-\sin(\theta)}\ .\label{eq:new_mt1}
\end{eqnarray}
By deriving \eqref{eq:new_mt1} w.r.t. $\theta$ one obtains
$$
\frac{\partial \cM(\theta)}{\partial\theta}=\frac{m_t-\sr_{t+1}}{1-\sin(\theta)}<0
$$
since $\sr_{t+1}>r_t\ge m_t$. Because the minimum feasible value of $\theta$ is $\underline\theta$ and by \eqref{eq:new_radius}, one has, under the
hypothesis \mbox{$\cM(\theta)=f_x(\theta)$,}
\begin{eqnarray}
&&\hspace{-10mm}\sup_{\theta:\cM(\theta)=f_x(\theta)}\cM(\theta)=\cM(\underline\theta)=\frac{r_{t+1}\,r_t}{\sr_{t+1}}\nonumber\\
&&\hspace{-7mm}=\frac{r_t}{\sr_{t+1}}\frac{m_t-\sr_{t+1}\sin(\underline\theta)}{1-\sin(\underline\theta)}=
r_t\frac{m_t-\sqrt{\sr_{t+1}^2-r_t^2}}{\sr_{t+1}-\sqrt{\sr_{t+1}^2-r_t^2}}\ .\label{eq:theta_opt1}
\end{eqnarray}
Let us now repeat the same reasoning for the case in which $\cM(\theta)=f_y(\theta)$. Notice that $\theta=\overline\theta$ satisfies such a condition,
yielding
\begin{eqnarray}
\cM(\overline\theta)&=&f_y(\overline\theta)=m_t-(\sr_{t+1}-r_{t+1})\sin(\overline\theta)=\frac{r_{t+1}\,m_t}{\sr_{t+1}}\nonumber\\
&<&r_{t+1}=f_x(\overline\theta)=\cR(\overline\theta)\ .\nonumber
\end{eqnarray}
Since $\cR(\theta)=f_x(\theta)=r_{t+1}$, by \eqref{eq:lem_xt}, one has
$$
r_{t+1}=r_t-(\sr_{t+1}-r_{t+1})\cos(\theta)
$$
which leads to
\begin{equation}\label{eq:new_radius2}
r_{t+1}=\frac{r_t-\sr_{t+1}\cos(\theta)}{1-\cos(\theta)}\ .
\end{equation}
Substituting \eqref{eq:new_radius2} into \eqref{eq:lem_yt}, after some algebra one gets
\begin{eqnarray}
\cM(\theta)&=&f_y(\theta)=r_t-(\sr_{t+1}-r_{t+1})\sin(\theta)\nonumber\\
&=&m_t-\frac{(\sr_{t+1}-r_t)\sin(\theta)}{1-\cos(\theta)}\ .\label{eq:new_mt2}
\end{eqnarray}
By deriving \eqref{eq:new_mt2} w.r.t. $\theta$ one obtains
$$
\frac{\partial \cM(\theta)}{\partial\theta}=\frac{\sr_{t+1}-r_t}{1-\cos(\theta)}>0\ .
$$
Thus,
\begin{eqnarray}
\sup_{\theta:\cM(\theta)=f_y(\theta)}\cM(\theta)&=&\cM(\overline\theta)=\frac{r_{t+1}\,m_t}{\sr_{t+1}}\nonumber\\
&=& \frac{m_t}{\sr_{t+1}}\frac{r_t-\sr_{t+1}\cos(\overline\theta)}{1-\cos(\overline\theta)}\nonumber\\
&=& m_t\frac{r_t-\sqrt{\sr_{t+1}^2-m_t^2}}{\sr_{t+1}-\sqrt{\sr_{t+1}^2-m_t^2}}\ .~~~~~~\label{eq:theta_opt2}
\end{eqnarray}
The result follows directly by \eqref{eq:theta_opt1} and \eqref{eq:theta_opt2}.
\end{proof}

In order to show that the MCLS leads to capture of the man in a finite number of moves, it is sufficient to prove that the strategy leads to $m_t\le 1$
for some finite $t$ (recall Proposition~\ref{prop:1_move}). In this respect, the worst situation for the lion is the one in which $m_t$ is maximized.
Lemma~\ref{lem:punti_estremi} states that for given $C_t$ and $\sr_{t+1}$, the lion location $L_{t+1}$ which maximizes the smallest element of $C_{t+1}$,
is one of the two points on the coordinate axes with distance $\sr_{t+1}$ from $C_t$. These points correspond to the extreme angles $\underline\theta$
and $\overline\theta$ for the direction of $C_t-L_{t+1}$ in Fig.~\ref{fig:dim_lem1}.

\begin{lemma}\label{lem:centro_su_bisettrice}
At a given time $t$, let $\widehat m_{t+1}$ be defined as in \eqref{eq:widehat_m}. Then,
$$
r_t^*\triangleq\arg \sup_{r_t\ge m_t} \widehat m_{t+1}=m_t
$$
and
\begin{equation}
m^*_{t+1}\triangleq\sup_{r_t\ge m_t}\widehat m_{t+1}=m_t\frac{m_t-\sqrt{\sr_{t+1}^2-m_t^2}}{\sr_{t+1}-\sqrt{\sr_{t+1}^2-m_t^2}}\label{eq:lem2}
\end{equation}
\end{lemma}
\begin{proof}
Recalling \eqref{eq:r1t_max}, let $r_t\triangleq\beta m_t$ with $\beta\ge1$. We want to find
$$
\beta^*=\arg~\sup_{\beta\ge1}~\widehat m_{t+1}\ .
$$
Let $\sr_{t+1}\triangleq \gamma r_t$, $\gamma>1$ and define $p=\sqrt{\sr_{t+1}^2-m_t^2}=m_t\sqrt{\beta^2\gamma^2-1}$, and
$q=\sqrt{\sr_{t+1}^2-r_t^2}=\beta m_t\sqrt{\gamma^2-1}$. For given $r_t$, $m_t$ and $\sr_{t+1}$, Lemma~\ref{lem:punti_estremi} states that
\begin{eqnarray}
\widehat m_{t+1}\!\!\!\!\!&=& \sup_{L_{t+1}:\|C_t-L_{t+1}\|=\sr_{t+1}}~ m_{t+1}\nonumber\\
&=&\!\!\!\!\max\left\{m_t\frac{\beta-\sqrt{\beta^2\gamma^2-1}}{\beta\gamma-\sqrt{\beta^2\gamma^2-1}} ,\,
m_t\frac{1-\beta\sqrt{\gamma^2-1}}{\gamma-\sqrt{\gamma^2-1}}\right\}.\nonumber\\
\end{eqnarray}
Let us consider the case
\begin{equation}\label{eq:lem_pr_csb}
\widehat m_{t+1} = m_t\frac{\beta-\sqrt{\beta^2\gamma^2-1}}{\beta\gamma-\sqrt{\beta^2\gamma^2-1}}\ .
\end{equation}
By deriving \eqref{eq:lem_pr_csb} w.r.t. $\beta$ one obtains
\begin{eqnarray}
\frac{\partial \widehat m_{t+1}}{\partial
\beta}&=&m_t\frac{\left(1-\frac{\beta\gamma^2}{\sqrt{\beta^2\gamma^2-1}}\right)\left(\beta\gamma-\sqrt{\beta^2\gamma^2-1}\right)}{\left(\beta\gamma-\sqrt{\beta^2\gamma^2-1}\right)^2}\nonumber\\
&&-m_t\frac{\left(\beta-\sqrt{\beta^2\gamma^2-1}\right)\left(\gamma-\frac{\beta\gamma^2}{\sqrt{\beta^2\gamma^2-1}}\right)}{\left(\beta\gamma-\sqrt{\beta^2\gamma^2-1}\right)^2}\nonumber\\
&=&m_t \frac{(\gamma-1)\left(
\sqrt{\beta^2\gamma^2-1}-\frac{\beta^2\gamma^2}{\sqrt{\beta^2\gamma^2-1}}\right)}{\left(\beta\gamma-\sqrt{\beta^2\gamma^2-1}\right)^2}\nonumber\\
&=& \frac{m_t(1-\gamma)}{\sqrt{\beta^2\gamma^2-1}\left(\beta\gamma-\sqrt{\beta^2\gamma^2-1}\right)^2}<0\ .\nonumber
\end{eqnarray}

Since $\ds\frac{\partial \widehat m_{t+1}}{\partial \beta}<0$, $\beta^*$ corresponds to its minimum feasible value, i.e., $\beta^*=1$, leading to
$r_t^*=m_t$.

By following the same reasoning, for the case in which
\begin{equation}\label{eq:lem_pr_csb2}
\widehat m_{t+1} = m_t\frac{1-\beta\sqrt{\gamma^2-1}}{\gamma-\sqrt{\gamma^2-1}}\ .
\end{equation}
one gets again $\ds\frac{\partial \widehat m_{t+1}}{\partial \beta}<0$, and then $r_t^*=m_t$. Expression \eqref{eq:lem2} is obtained by direct
substitution into \eqref{eq:widehat_m}.

\end{proof}

Lemma~\ref{lem:centro_su_bisettrice} states that, for a given $m_t$, the center $C_t$ which (potentially) leads to the maximum $m_{t+1}$ at the
subsequent step is $C_t=[m_t,m_t]'$. This is instrumental to define a bound to the evolution of $m_t$.

\begin{theorem}\label{th:convergence}
Let $m_t>1$ and let the lion play the MCLS. Then, for any possible man strategy, one has
\begin{equation}\nonumber
\ds m_{t+1}\le\frac{m_t(m_t-1)}{\sqrt{1+m_t^2}-1}\ .  
\end{equation}
\end{theorem}
\begin{proof}
By Lemma~\ref{lem:punti_estremi} and Lemma~\ref{lem:centro_su_bisettrice}, one has
\begin{eqnarray}
m_{t+1}^*&=& \sup_{r_t\ge m_t}~\sup_{L_{t+1}:\|C_t-L_{t+1}\|=\sr_{t+1}} m_{t+1}\nonumber\\
&=&                       %
\sup_{r_t\ge m_t} \widehat m_{t+1}%
=m_t\frac{m_t-\sqrt{\sr_{t+1}^2-m_t^2}}{\sr_{t+1}-\sqrt{\sr_{t+1}^2-m_t^2}}\ .~~~~~~\label{eq:lem_conv}
\end{eqnarray}
Let us define $\hat r=\sr_{t+1}/m_t$. By substituting into \eqref{eq:lem_conv}, one has
\begin{equation}\label{eq:lem_conv2}
m_{t+1}^*=m_t\frac{1-\sqrt{\hat{r}^2-1}}{\hat{r}-\sqrt{\hat{r}^2-1}}\ .
\end{equation}
By deriving $m_{t+1}^*$ w.r.t. $\hat{r}$, one has
$$
\frac{\partial m_{t+1}^*}{\partial \hat{r}}=m_t\frac{\hat{r}-1-\sqrt{\hat{r}^2-1}}{\left(\hat{r}-\sqrt{\hat{r}^2-1} \right)^2}
$$
which vanishes for $\hat r=1$. It can be easily checked that $\hat{r}=1$ corresponds to a maximum and $\frac{\partial m_{t+1}^*}{\partial \hat{r}}<0$,
$\forall \hat r>1$, i.e. $\forall \sr_{t+1}>m_t$. Since by \eqref{eq:r1t_max} one has that $\sr_{t+1}^2\ge r_t^2+1= m_t^2+1$, the maximum value for
$m_{t+1}^*$ is achieved for $\sr_{t+1}=\sqrt{m_t^2+1}$. By substituting in \eqref{eq:lem_conv2} one has
\begin{eqnarray}
m_{t+1}^*&=&m_t\frac{m_t-\sqrt{m_t^2+1-m_t^2}}{\sqrt{m_t^2+1}-\sqrt{m_t^2+1-m_t^2}}\nonumber\\
&=&m_t\frac{m_t-1}{\sqrt{m_t^2+1}-1}\nonumber
\end{eqnarray}
which concludes the proof.
\end{proof}

A direct consequence of Theorem~\ref{th:convergence} is that MCLS leads to capture of the man in a finite number of moves. An upper bound to such a
number is now derived.

Let us consider the recursion
\begin{equation}\label{eq:mtpiu1bis}
b_{t+1}=\frac{b_t(b_t-1)}{\sqrt{1+b_t^2}-1}\triangleq g(b_t)\ .
\end{equation}
Let us fix $b_0=m_0>1$. Since the function $g(b_t)$ is monotone increasing for $b_t>1$, by Theorem~\ref{th:convergence} one has that, if $m_t\le b_t$,
then
$$
m_{t+1}\le g(m_t)\le g(b_t)=b_{t+1}\ .
$$
Therefore, recursion \eqref{eq:mtpiu1bis} returns an upper bound of $m_{t}$, for all $t$. By Proposition~\ref{prop:1_move}, if $m_t\le1$ then the game
ends at the next move. Since $g(b_t)<0$ when $b_t<1$, an upper bound to the maximum number of moves before the game ends can be computed as follows
$$
N_{max}^{MCLS}=\min\{t\in\nn:~b_t <0\}\ .
$$
Notice that $N_{max}^{MCLS}$ is a function of $m_0$, although it seems difficult to express this dependence explicitly. Clearly, $N_{max}^{MCLS}$ can be
numerically computed by recursively evaluating $b_{t}$ in \eqref{eq:mtpiu1bis}. In Fig.~\ref{fig:upper_bounds}, $N_{max}^{MCLS}$ and
$N_{max}^{FCLS}=\lceil m_0^2\rceil$ are compared for $m_0\in [1,\,10]$.

\begin{figure}[htb]
\centering %
\psfrag{m0}[c]{$m_0$}%
\psfrag{yy}[]{\scriptsize{Upper bound on the number of moves}}%
\includegraphics[width=.6\columnwidth]{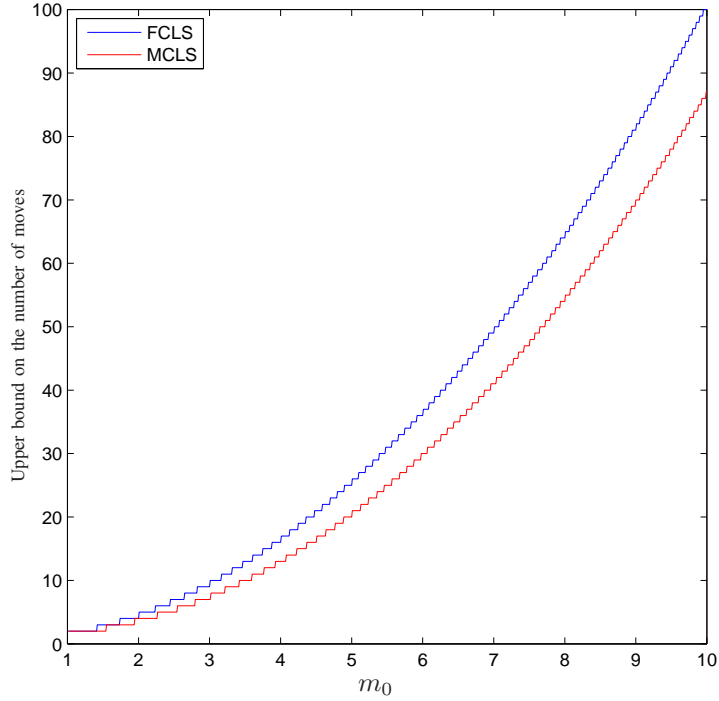}
\caption{Upper bounds on the number of moves for different values of $m_0$. Comparison between $N_{max}^{FCLS}$ (blue) and M$N_{max}^{MCLS}$CLS (red).}
\label{fig:upper_bounds}
\end{figure}

From Fig.~\ref{fig:upper_bounds} it is apparent that $N_{max}^{MCLS}\le N_{max}^{FCLS}$, where equality holds only for small values of $m_0$ (due to the
discretization introduced by the fact that the number of moves must be integer). This fact is proved in the following theorem for every $m_0>1$.

\begin{theorem}
Let $m_0>1$. Then, $N_{max}^{MCLS}\le N_{max}^{FCLS}=\lceil m_0^2\rceil$.
\end{theorem}
\begin{proof}
System \eqref{eq:mtpiu1bis} can be rewritten as
$$
b_{t+1} = \frac{(b_t-1)(\sqrt{1+b_t^2}+1)}{b_t}
$$
which leads to
\begin{eqnarray}
b^2_{t+1}\!\!\!\! &=&\!\!\! \ds{\frac{(b_t-1)^2(b_t^2+2+2\sqrt{1+b_t^2})}{b^2_t} }\nonumber\\
&=&\!\!\!\ds{b^2_t- \frac{2b^3_t-3b_t^2+4b_t-2-2(b_t-1)^2\sqrt{1+b_t^2}}{b^2_t}}\ .\nonumber\\\label{eq:b2}
\end{eqnarray}

For a given $m_0$, let us consider the system
\begin{equation}\label{eq:m_0sgall}
m_{t+1}^2=m_t^2-1\ .
\end{equation}
Clearly, $\min\{t\in\nn:\,m_t<0\}=\lceil m_0^2\rceil=N_{max}^{FCLS}$. Therefore, to prove the theorem it is sufficient to show that system \eqref{eq:b2}
decays to zero always faster than \eqref{eq:m_0sgall}, which amounts to show that
$$
\frac{2b^3_t-3b_t^2+4b_t-2-2(b_t-1)^2\sqrt{1+b_t^2}}{b^2_t} > 1
$$
for all $b_t>1$. This easily follows from standard calculus arguments.
\end{proof}

\section{Conclusions}\label{sec:conclusion}

A new lion strategy has been devised for the discrete-time version of the lion and man problem. This solution dominates the one proposed by Sgall in
\cite{sgall01} in terms of maximum number of moves required to guarantee man capture.

An interesting feature of the proposed approach is that the upper bound on the number of moves does not seem to be tight. Indeed, for randomly chosen
initial conditions of lion and man, numerical simulations show that the actual number of steps in which the lion reaches the man turns out to be much
smaller than that predicted by the bound. Unfortunately, the optimal man strategy for counteracting the proposed lion algorithm, for generic initial
conditions, is still an open problem. It is expected that such a result would allow one to significantly improve the upper bound on the number of moves.

We believe that the proposed result is helpful in all the contexts in which the lion and man problem solution is used as a building block within more
complex strategies for pursuit-evasion games, like in \cite{isler2005}. The application of the new lion algorithm in these problems and the evaluation of
its benefits is the subject of ongoing research.

\balance

\bibliographystyle{IEEEtran}      
\bibliography{lion_and_man}

\end{document}